\documentclass[12 pt]{article}%
\usepackage{amsmath, amsfonts, amsthm, color,latexsym}
\usepackage{amsmath}
\usepackage{amsfonts}
\usepackage{amssymb}
\usepackage{color, soul}
\usepackage{graphicx}%
\providecommand{\U}[1]{\protect\rule{.1in}{.1in}}
\newtheorem{theorem}{Theorem}[section]
\newtheorem{proposition}[theorem]{Proposition}

\newtheorem{example}[theorem]{Example}

\newtheorem{lemma}[theorem]{Lemma}
\newtheorem{final remark}[theorem]{Final Remark}
\newtheorem{definition}[theorem]{Definition}
\begin{document}

\title{Absolutely summing multilinear operators on $\ell_p$ spaces}
\date{}
\author{Oscar Blasco\thanks{Supported by Ministerio de Ciencia e Innovaci\'{o}n
MTM2011-23164}\,, Geraldo Botelho\thanks{Supported by CNPq Grant
302177/2011-6 and Fapemig Grant PPM-00326-13.}\,, Daniel Pellegrino\thanks{Supported by
INCT-Matem\'atica, PROCAD-NF-Capes, CNPq Grants 620108/2008-8 (Ed. Casadinho)
and 301237/2009-3.}~ and Pilar Rueda\thanks{Supported by Ministerio de Ciencia
e Innovaci\'{o}n  MTM2011-22417.\hfill\newline2010 Mathematics Subject
Classification. Primary 46G25; Secondary 47B10, 47L22. \newline Keywords: absolutely
summing, almost summing, multilinear mapping, Banach spaces.}}
\maketitle

\begin{abstract}
We prove new summability properties for multilinear operators on $\ell_p$ spaces. An important tool for this task is a better understanding of the interplay between almost summing and absolutely summing multilinear operators.
\end{abstract}

\section{Introduction}

This work is the outgrowth of several papers by the authors and other researchers (see, e.g., \cite{BBPR, nosso, Archiv, Junek-preprint, michels, Port, Indag, ursula, sevilla, Jarchow, Junek, advances, pop} and references therein) on summability properties of multilinear mappings between Banach spaces.

As a consequence of the successful theory of absolutely summing linear operators, the study of summability properties of multilinear mappings focuses on mappings that improve the summability of series. Historical accounts of this multilinear theory can be found in \cite{BBPR, advances}.

As expected, at the heart of the theory lies the investigation of the interplay between different types of summability of multilinear mappings. For example, \cite{davidarchiv} is a comparative study of three different classes of absolutely summing multilinear mappings.  Having in mind that the interplay between almost summability and absolute summability is quite profitable in the linear case (see \cite[Chapter 12]{Diestel}), in this paper we will explore the connections of almost summing and absolutely summing multilinear mappings and apply them to obtain summability properties for multilinear operators on $\ell_p$ spaces.

The first natural issue is the well know fact that absolutely $p$-summing linear operators are almost summing for every $p$. In some situations, for instance in the Pietsch Domination Theorem, the multilinear mappings that play the role of the absolutely summing linear operators are the dominated ones. So it is a natural question if dominated multilinear mappings are almost summing. In Section 3 we prove that dominated multilinear mappings  actually satisfy a condition stronger than being almost summing.

One of the trends of the theory is the identification of coincidence situations, that is, a situation where $E_1, \ldots, E_n,F$ are Banach spaces and every continuous $n$-linear mapping $A \colon E_1 \times \cdots \times E_n \longrightarrow F$ enjoys a certain summability property. Probably the first result in this line is the Defant--Voigt Theorem (see \cite[Corollary 3.2]{Port} or \cite{BBPR}, where an improved version can be found), which states that multilinear forms are always $(1;1, \ldots,1)$-summing. Several other multilinear coincidence results can be found in, e.g., \cite[Theorem 3.3]{bote97}, \cite[Proposition 2.1]{pams} and \cite[Theorem~3.7]{BBPR}. The core of Section 4 is the obtainment of new coincidence results for absolutely summing multilinear functionals on $\ell_p$ spaces. According to the main idea of this paper, these new absolutely summing coincidence results will follow from a general result (Theorem \ref{as}) concerning almost summing multilinear mappings. We prove for instance that for any
$1\leq p \leq2$, continuous bilinear forms $A$ on $\ell
_{p}\times F$, where $F'$ has type 2, are
absolutely $(p;2,1)$-summing. Moreover,  $A$ is
absolutely $(r_{p};r_{p},r_{p})$-summing for any $1\leq r_{p}\leq\frac
{2p}{3p-2}$. In particular, if $1\leq p\leq2$ and $1<q\leq2$, then bilinear
forms on $\ell_{p}\times\ell_{q}$ are $(p;2,1)-$summing and bilinear forms on $\ell_{1}\times\ell_{q}$ are $(r;r,r)-$summing for
$1\leq r\leq2.$

\section{Notation and background}

\label{not}

All Banach spaces are considered over the scalar field $\mathbb{K}=\mathbb{R}$
or $\mathbb{C}$. Given a Banach space $E$, let $B_{E}$ denote the closed unit
ball of $E$ and $E^{\prime}$ its topological dual.

Let $p > 0$. By $\ell_{p}(E)$ we denote the ($p$-)Banach space of all absolutely
$p$-summable sequences $(x_{j})_{j=1}^{\infty}$ in $E$ endowed with its usual
$\ell_{p}$-norm ($p$-norm if $0 < p < 1$).
%such that $\|(x_{j})_{j=1}^{\infty}%
%\|_{p}:=(\sum_{j=1}^{\infty}\|x_{j}\|^{p})^{\frac1p}<\infty$ with the norm
%$\|.\|_{p}$.
Let $\ell_{p}^{w}(E)$ be the space of those sequences $(x_{j})_{j=1}^{\infty}$
in $E$ such that $(\varphi(x_{j}))_{j=1}^{\infty}\in\ell_{p}$ for every
$\varphi\in E^{\prime}$ endowed with the norm ($p$-norm if $0 < p < 1$)
\[
\|(x_{j})_{j=1}^{\infty}\|_{\ell_{p}^{w}(E)}=\sup_{\varphi\in B_{E^{\prime}}%
}\left(\sum_{j=1}^{\infty}|\varphi(x_{j})|^{p} \right)^{\frac1p}.
\]
Let $\ell_{p}^{u}(E)$ denote the closed subspace of $\ell_{p}^{w}(E)$ formed
by the sequences $(x_{j})_{j=1}^{\infty}\in\ell_{p}^{w}(E)$ such that
$\lim_{k\rightarrow\infty}\Vert(x_{j})_{j=k}^{\infty}\Vert_{\ell_{p}^{w}%
(E)}=0.$

Let $E_{1},\ldots,E_{n},E,F$ be Banach spaces. The Banach space of all
continuous $n$-linear mappings from $E_{1}\times\cdots\times E_{n}$ into $F$
is denoted by $\mathcal{L}(E_{1},\ldots,E_{n};F)$ and endowed with the usual sup
norm. We simply write $\mathcal{L}(^{n}E;F)$ when $E_{1}=\cdots=E_{n}=E$.

For $0< p,p_{1},p_{2},\ldots,p_{n}\leq\infty$ , we assume that $\frac{1}%
{p}\leq\frac{1}{p_{1}}+\cdots+\frac{1}{p_{n}}$. An $n-$linear mapping
$A\in{\mathcal{L}}(E_{1},\ldots,E_{n};F)$ is \textit{absolutely $(p;p_{1}%
,p_{2},\ldots,p_{n})$-summing} if there is $C>0$ such that
\[
\Vert(A(x_{j}^{1},x_{j}^{2},\ldots,x_{j}^{n}))_{j=1}^k\Vert_{p}\leq C\prod
_{i=1}^{n}\Vert(x_{j}^{i})_{j=1}^k\Vert_{\ell_{p_{i}}^{w}(E_{i})}%
\]
for all finite families of vectors $x_{1}^{i}, \ldots, x_k^i  \in E_{i}$,  $i=1,2,\ldots,n$.
The infimum of such $C>0$ is called the $(p;p_{1},\ldots,p_{n} )$-summing norm
of $A$ and is denoted by $\pi_{(p;p_{1},\ldots,p_{n})}(A)$. Let $\Pi_{(p;p_{1}
,p_{2},\ldots,p_{n})}(E_{1},\ldots,E_{n};F)$ denote the space of all
absolutely $(p;p_{1},p_{2},\ldots,p_{n})$-summing $n$-linear mappings from
$E_{1}\times\cdots\times E_{n}$ to $F$ endowed with the norm $\pi
_{(p;p_{1}\ldots,p_{n})}$.

It is well known that we can replace $\ell_{p_{k}}^{w}(E_{k})$ with $\ell
_{p_{k}}^{u}(E_{k})$ in the definition of absolutely summing mappings.

Absolutely summing mappings fulfill the following inclusion result, which
appears in \cite[Proposition 3.3]{Thesis} (see also \cite{BBPR}):

\begin{theorem}\label{inclusiontheorem}{\rm (Inclusion Theorem)} \ Let $0< q\leq
p\leq\infty$, $0< q_{j}\leq p_{j}\leq\infty$ for all $j=1,\ldots,n$. If
$\frac{1}{q_{1}}+\cdots+\frac{1}{q_{n}}-\frac{1}{q}\leq\frac{1}{p_{1}}+
\cdots+\frac{1}{p_{n}}-\frac{1}{p}$, then
\[
\Pi_{(q;q_{1},\ldots,q_{n})}(E_{1},\ldots,E_{n};F)\subseteq\Pi_{(p;p_{1}%
,\ldots,p_{n})}(E_{1},\ldots,E_{n};F)
\]
and $\pi_{(p;p_{1},\ldots,p_{n})}\leq\pi_{(q;q_{1} ,\ldots,q_{n})}$ for all Banach spaces $E_1, \ldots, E_n,F$.
\end{theorem}

\medskip If $\frac{1}{p} = \frac{1}{p_1} + \cdots + \frac{1}{p_n}$, absolutely $(p;p_1, \ldots, p_n)$-summing $n$-linear mappings
are usually called \textit{$(p_1, \ldots, p_n)$-dominated}. They satisfy the following
factorization result (see \cite[Theorem 13]{Pie}):

\begin{theorem}\label{factorizationtheorem}{\rm (Factorization Theorem)} A multilinear mapping $A \in {\cal L}(E_{1},\ldots,E_{n};F)$ is $(p_1, \ldots, p_n)$-dominated if and only if there are Banach
spaces $G_{1}, \ldots, G_{n}$, operators $u_{j} \in\Pi_{p_j}(E_{j};G_{j})$, $j = 1, \ldots, n$, and
$B \in\mathcal{L}(G_{1},\ldots,G_{n};F)$ such that $\label{fact}A = B \circ(u_{1}, \ldots, u_{n})$.
\end{theorem}

As mentioned in the introduction, the next theorem is probably the first coincidence result for multilinear mappings:

\begin{theorem}{\rm (Defant--Voigt Theorem)} Let $n \geq 2$ and $E_1, \ldots, E_n$ be Banach spaces. Then ${\cal L}(E_1, \ldots, E_n;\mathbb{K}) = \Pi_{(1;1,\ldots,1)}(E_1, \ldots, E_n;\mathbb{K})$ isometrically.
\end{theorem}

We denote by $Rad(E)$ the space of sequences $(x_{j})_{j=1}^{\infty}$ in $E$
such that
\[
\Vert(x_{j})_{j=1}^{\infty}\Vert_{Rad(E)}=\sup_{n\in\mathbb{N}}\left\|\sum
_{j=1}^{n}r_{j}x_{j}\right\|_{L^{2}([0,1],E)}<\infty,
\]
where $(r_{j})_{j\in\mathbb{N}}$ are the Rademacher functions on $[0,1]$
defined by $r_{j}(t)={\rm sign}(\sin2^{j}\pi t)$. If instead of the $L^{2}([0,1],E)$
norm one considers in the definition of $Rad(E)$ the $L^{p}([0,1],E)$ norm,
$1\leq p<\infty$, one gets equivalent norms on $Rad(E)$ as a consequence of
Kahane's inequalities (see \cite[p. 211]{Diestel}). These norms, when needed,
will be denoted by $\|\cdot\|_{Rad_{p}(E)}$.

Recall also that a linear operator $u\colon E\longrightarrow F$ is said to be
\textit{almost summing} if there is a $C>0$ such that%
\[
\Vert\left(  u(x_{j})\right)  _{j=1}^{m}\Vert_{Rad(F)} \leq C\left\Vert
(x_{j})_{j=1} ^{m}\right\Vert _{\ell^{w}_{2}(E)}%
\]
for any finite set of vectors $\{x_{1},\ldots,x_{m}\}$ in $E$. The space of
all almost summing linear operators from $E$ to $F$ is denoted by $\Pi
_{as}(E;F)$ and the infimum of all $C>0$ fulfilling the above inequality is
denoted by $\|u\|_{as}$. Note that this definition differs from the
definition of almost summing operators given in \cite[p. 234]{Diestel} but
coincides with the characterization which appears a few lines after that
definition (yes, the definition and the stated characterization are not
equivalent). Since the proof of \cite[Proposition 12.5]{Diestel} uses the
characterization (which is our definition), we can conclude that every
absolutely $p$-summing linear operator, $1 \leq p < + \infty$, is almost summing.

The concept of almost summing multilinear mapping was considered in
\cite{Nach, Archiv} and reads as follows: A multilinear map $A\in{\mathcal{L}%
}(E_{1},\ldots,E_{n};F)$ is said to be \textit{almost summing} if there exists
$C>0$ such that
\begin{equation}
\label{almostsumming}\| \left(  A(x^{1}_{j},\ldots,x^{n}_{j})\right)  _{j=1}%
^{m}\|_{Rad(F)}\le C\prod_{i=1}^{n} \Vert(x_{j}^{i})_{j=1}^{m}\Vert_{\ell
_{2}^{w}(E_{i})}%
\end{equation}
for any finite set of vectors $(x^{i}_{j})_{j=1}^{m} \subseteq E_{i}$ for
$i=1,\ldots,n$. We write $\Pi_{as}(E_{1},\ldots,E_{n};F)$ for the space of almost
summing multilinear maps, which is
%%@
endowed with the norm
\[
\|A\|_{as}:= \inf\{C > 0\ \mbox{such\ that\ (\ref{almostsumming})\ holds} \}.
\]

For $n\ge 1$ and $A\in{\mathcal{L}}(E_{1},\ldots,E_{n};F)$,
$$\widehat {A}\colon \ell_\infty(E_1)\times \cdots \times \ell_\infty(E_n)\longrightarrow \ell_\infty(F)~,~
\widehat{A}((x_{j}^{1})_{j=1}^{\infty},\ldots,(x_{j}^{n})_{j=1}^{\infty
}):=(A(x_{j}^{1},\ldots,x_{j}^{n}))_{j=1}^{\infty},$$
is a bounded $n$-linear mapping. Given subspaces $X_{i}\subseteq \ell_\infty(E_i)$ for $1\le i\le n$ and $Y\subseteq \ell_\infty(F)$, we say that $\widehat{A}\colon X_{1}\times\cdots\times X_{n}\longrightarrow Y$ is bounded -- or, equivalently, $\widehat{A} \in {\cal L}(X_{1}, \ldots, X_{n}; Y)$ --
if the restriction of $\widehat A$ to $X_1 \times \cdots \times X_n$ is a well defined (hence $n$-linear) continuous $Y$-valued mapping.

\section{The interplay between almost summing and absolutely summing multilinear mappings}

In next section we use almost summing multilinear mappings to identify a number of situations where all multilinear functionals on $\ell_p$ spaces are absolutely summing. These coincidence results will be consequences of the next result, which asserts that if a $k$-linear mapping, $1 \leq k < n$, associated to the $n$-linear form $A$ is almost summing, then $A$ is not only absolutely $(1;1, \ldots, 1)$-summing but absolutely $(1;2, \ldots,2,1\ldots 1)$-summing.

\begin{definition}\rm Let $E_1, \ldots, E_n,F$ be Banach spaces and $A \colon E_1 \times \cdots \times E_n \longrightarrow F$ be a continuous $n$-linear mapping. For $1 \leq k < n$, the $k$-linear mapping $A_k$ associated to $A$ is given by
\[
A_k\colon E_1 \times \cdots \times E_k \longrightarrow {\cal L}(E_{k+1},\ldots, E_n;F)~,~A_{k}(x_{1},\ldots,x_{k})(x_{k+1},\ldots,x_{n})=A(x_{1},\ldots,x_{n}).
\]
It is clear that $A_k \in {\cal L}(E_1,\ldots, E_k;{\cal L}(E_{k+1},\ldots, E_n;F))$ and $\|A_k\| = \|A\|$.
\end{definition}

\begin{theorem}
\label{as} Let $1\le k<n$ and $A\in{\mathcal{L}}(E_{1},\ldots,E_{n};\mathbb{K})$
be such that
\[
A_{k}\in\Pi_{as}(E_{1},\ldots,E_{k};{\mathcal{L}}(E_{k+1},\ldots,E_{n}%
;\mathbb{K})).
\]
Then,
\[
\widehat A \colon \ell_{2}^{w}(E_{1})\times\cdots\times\ell_{2}^{w}(E_{k})\times
Rad(E_{k+1}) \times\cdots\times Rad(E_{n}) \longrightarrow\ell_{1},%
\]
is bounded. Moreover $\|\widehat A\|\le\|A_{k}\|_{as}.$

 In particular $A\in \Pi_{(1;2,\ldots 2, 1,\ldots, 1)}(E_{1},\ldots,E_{n};\mathbb{K})$.

\end{theorem}

\begin{proof} Let $(x_{j}^{i})_{j=1}^\infty$ be a finite sequence in $E_{i}$,  $i=1,\ldots
,n.$ Take a scalar sequence $(\alpha_{j})_{j=1}^\infty$, let $A^j_k=
A_k(x_j^1,x_j^2,\ldots,x_j^k)$ and define
$$f_\alpha(t_k)=\sum_{j=1}^\infty\alpha_j A^j_kr_j(t_k);~f_i(t_i)=\sum_{j=1}^\infty r_j(t_i)x_j^i,~ i=k+1,\ldots,n-1;~{\rm and}$$
$$f_{n}(t_{k},\ldots,t_{n-1})=\sum_{j=1}^\infty r_j(t_{k})\cdots r_j(t_{n-1})x_j^n, \quad t_{k},\ldots,t_{n-1} \in
[0,1].$$
The orthogonality of the
Rademacher system  shows that
\begin{align*}
&\sum_{j=1}^\infty A(\alpha_jx_j^1,\ldots,x_j^n)=\sum_{j=1}^\infty
A_k(\alpha_jx_j^1,\ldots,x_j^k)(x_j^{k+1},\ldots,x_j^n) =  \sum_{j=1}^\infty \alpha_j
A_k^j(x_j^{k+1},\ldots,x_j^n)
\\&=\int_0^1\cdots\int_0^1
f_\alpha(t_k)( f_{k+1}(t_{k+1}),\ldots, f_{n-1}(t_{n-1}),f_n(t_{k},\ldots,t_{n-1}))dt_k \cdots dt_{n-1}\\
&\le\int_0^1\cdots\int_0^1 \Big(\int_0^1\| f_\alpha(t_k)\|\cdot
\|f_n(t_{k},\ldots,t_{n-1})\|dt_k\Big) \|
f_{k+1}(t_{k+1})\|\cdots \|f_{n-1}(t_{n-1})\|\  dt_{k+1}\cdots dt_{n-1} \\
&\le\|A_k\|_{as}\prod_{i=1}^k\|(x_j^i)_{j=1}^\infty\|_{\ell_2^w(E_i)}\\
&.\int_0^1\cdots\int_0^1 \Big(\int_0^1 \|f_{n}(t_k,\ldots,t_{n-1})\|^2dt_{k}\Big)^{1/2}
\| f_{k+1}(t_{k+1})\|\cdots \|f_{n-1}(t_{n-1})\| dt_{k+1}\cdots dt_{n-1} \\&\le   \|A_k\|_{as}
\prod_{i=1}^k\|(x_j^i)_{j=1}^\infty\|_{\ell_2^w(E_i)}
\Big(\prod_{i=k+1}^{n-1}\|(x_j^i)_{j=1}^\infty\|_{Rad_1(F)}\Big)\|( x_j^n)_{j=1}^\infty\|_{Rad_2(F)}.
\end{align*}
The result follows.
\end{proof}

\label{dominated}

It is well known that $p-$absolutely summing linear operators are almost summing, more precisely (see \cite[Proposition 12.5]{Diestel}): \begin{equation}\label{lc}\bigcup_{p>0}\Pi
_{p}(E;F)\subseteq\Pi_{as}(E;F).\end{equation}

%Recall that $A\in \Pi_{as}(E_{1},\ldots,E_{n};F)$ whenever $\hat A \colon \ell_2^w(E_1)\times \cdots\times\ell_2^w(E_n)\longrightarrow Rad(F)$ is bounded.
In the multilinear setting, for a Hilbert space $H$, clearly
  $$\Pi_{as}(E_{1},\ldots,E_{n};H)= \Pi_{(2;2,\ldots,2)}(E_{1},\ldots,E_{n};H),$$
  because
$Rad(H)=\ell_{2}(H)$; and the corresponding inclusions hold whenever $F$ has
type $p$ or cotype $q$.

In view of \cite[Theorem 4.1]{Nach}, a multilinear version of (\ref{lc}) asserting that dominated multilinear mappings are almost summing is expected. Next we give a short proof of this fact, but the aim of this section is to go a bit further.

\begin{proposition} Let $M_n=\left\{(p_1,\ldots,p_n,p)\in \mathbb R_+^{n+1}:  \frac{1}{p}=
\sum_{i=1}^n\frac{1}{p_i}\right\}$ and $E_1, \ldots, E_n, F$ be Banach spaces. Then
\[
\bigcup_{(p_1,\ldots,p_n,p)\in M_n}\Pi_{(p;p_1,\ldots,p_n)}(E_{1},\ldots, E_{n};F)\subseteq\Pi
_{as}(E_{1},\ldots, E_{n};F).
\]
\end{proposition}
\begin{proof}
Using (\ref{lc})
and the Factorization Theorem \ref{fact} -- see also \cite[Theorem 4.1]{Nach} -- it is not difficult to see that
\[\bigcup_{p>0}\Pi_{(p/n;p\ldots,p)}(E_{1},\ldots, E_{n};F)\subseteq\Pi
_{as}(E_{1},\ldots, E_{n};F).
\]
Given $(p_1,\ldots,p_n,p)\in M$, using the Inclusion Theorem \ref{inclusiontheorem} and letting $p_0=\max\{p_i, 1\le i\le n\}$, we have
$$\Pi_{(p;p_1,\ldots,p_n)}(E_{1},\ldots, E_{n};F)
\subseteq \Pi_{(\frac{p_0}{n};p_0,\ldots,p_0)}(E_{1},\ldots, E_{n};F),$$
which gives the result.
\end{proof}

We now show that by replacing $Rad(F)$ with a bigger space $Rad^{(2)}(F)$ we get a variation of the above result,  with a kind of multiple summation.

\begin{definition}\rm
For a Banach space $F$, by $Rad^{(2)}(F)$ we denote the space of sequences $(x_{i,j})_{i,j\ge 0}\subseteq  F$ such that
$$\|(x_{i,j})\|_{Rad^{(2)}(F)}=\sup_{m\ge 0} \left(\int_0^1 \left\|\sum_{i,j=0}^m x_{i,j}r_i(t)r_j(t)\right\|^2dt\right)^{1/2}<\infty.$$
\end{definition}

We give an example where the number $\|(x_{i,j})\|_{Rad^{(2)}(F)}$ can be explicitly computed:

\begin{example}\rm Given $d\in \mathbb N$, for $F=\mathbb C^d$ we have
$$\|(x_{i,j})\|_{Rad^{(2)}(F)}= \sup_{m}\left(\sum_{i=0}^m \|x_{i,i}\|^2+2\sum_{i<j}^m Re(\langle x_{i,i},x_{j,j}\rangle)+ \sum_{i<j}^m \|x_{i,j}+x_{j,i}\|^2\right)^{1/2}.$$
Indeed, given $m\ge 0$, for $\langle x,y\rangle =\sum_{j=1}^d x(j)\overline{y(j)}$,
\begin{align*}
\int_0^1 \left\|\sum_{i,j=0}^m x_{i,j}r_i(t)r_j(t)\right\|^2dt&= \int_0^1 \left\|\sum_{i=0}^m x_{i,i}+\sum_{i<j}^m (x_{i,j}+ x_{j,i}) r_i(t)r_j(t)\right\|^2dt\\
&=  \left\langle \sum_{i=0}^m x_{i,i}, \sum_{i'=0}^m x_{i',i'}\right\rangle  \\&+
\int_0^1
\left\langle \sum_{ i<j}^m (x_{i,j}+ x_{j,i})r_i(t)r_j(t),\sum_{i'=0}^m x_{i',i'}\right\rangle   dt \\
&+
\int_0^1 \left\langle\sum_{i'=0}^m x_{i',i'},  \sum_{ i<j}^m (x_{i,j}+ x_{j,i}) r_i(t)r_j(t)\right\rangle dt \\
&+
\int_0^1 \left\langle \sum_{ i<j}^m (x_{i,j}+ x_{j,i}) r_i(t)r_j(t),\sum_{ i'<j'}^m (x_{i',j'}+ x_{j',i'}) r_{i'}(t)r_{j'}(t)\right\rangle  dt.
\end{align*}
The desired formula now follows using that $\int_0^1 r_i(t)r_j(t) r_{i'}(t)r_{j'}(t) dt=0$ unless $i=i'$ and $j=j'$ (see \cite[p.\,10]{Diestel}) and the fact
$\langle x,y\rangle +\langle y,x\rangle = 2Re(\langle x,y\rangle )$.
 \end{example}

\begin{definition}\rm We say that a bilinear map $A:E_1\times E_2\longrightarrow F$ is {\it bilinear almost summing} if  there exists
$C>0$ such that
\begin{equation}
\label{bilinearalmostsumming}\left\| \left(  A(x^{1}_{i},x^{2}_{j})\right)  _{i,j=0}%
^{m}\right\|_{Rad^{(2)}(F)}\le C\prod_{k=1}^{2} \Vert(x_{j}^{k})_{j=0}^{m}\Vert_{\ell
_{2}^{w}(E_{k})}%
\end{equation}
for any finite set of vectors $(x^{k}_{j})_{j=0}^{m} \subseteq E_{k}$ for
$k=1,2$. We write $\Pi_{bas}(E_{1},E_{2};F)$ for the space of bilinear almost
summing multilinear maps, which is
%%@
endowed with the norm
\[
\|A\|_{bas}:= \inf\{C > 0\ \mbox{such\ that\ (\ref{bilinearalmostsumming})\ holds} \}.
\]
\end{definition}

 \begin{proposition} \label{relation} $\Pi_{bas}(E_{1},E_{2};F)\subseteq \Pi_{as}(E_{1},E_{2};F)$
\end{proposition}

\begin{proof} Let $A\in \Pi_{bas}(E_{1},E_{2};F)$ and $(x^{k}_{j})_{j=0}^{m} \subseteq E_{k}$ for
$k=1,2$. Denote $A(x^1_i,x^2_j)=x_{i,j}\in F$ for $0\le i,j\le m$.
%\corr{Then  $ \|(x_{i,i})\|_{Rad(F)}\le \|(x_{i,j})\|_{Rad^{(2)}(F)}.$}{}
By the orthogonality of the Rademacher functions we have
%$$ \sum_{i=0}^m x_{i,i}=\int_0^1 \Big(\sum_{i=0}^m x_{i,i}+\sum_{i<j} (x_{i,j}+ x_{j,i}) r_i(t)r_j(t)\Big) dt$$
%one gets that
\begin{align*}\left\|\sum_{i=0}^m x_{i,i}\right\|&=\left\|\int_0^1 \sum_{i,j=0}^m x_{i,j}r_i(t)r_j(t)dt\right\|\le \int_0^1 \left\|\sum_{i,j=0}^m x_{i,j}r_i(t)r_j(t)\right\|dt\\&\le \|(x_{i,j})_{i,j=0}^m\|_{Rad^{(2)}(F)}\le \|A\|_{bas}\prod_{k=1}^{2} \Vert(x_{j}^{k})_{j=0}^{m}\Vert_{\ell
_{2}^{w}(E_{k})}.
\end{align*}
Therefore, replacing $x^1_i$ with $r_i(t)x^1_i$ one has
$$ \left\|\sum_{i=0}^m A(x^1_i,x^2_i)r_i(t)\right\|\le \|A\|_{bas}\cdot\prod_{k=1}^{2} \Vert(x_{j}^{k})_{j=0}^{m}\Vert_{\ell
_{2}^{w}(E_{k})},~ t\in [0,1].$$
Now integrating over $[0,1]$ one gets $\|A\|_{as}\le C\|A\|_{bas}.$
\end{proof}

\begin{theorem}
\label{bas} Let $M=\left\{(p_1,p_2,2)\in \mathbb R_+^3:  \frac{1}{2}=
\frac{1}{p_1}+ \frac{1}{p_2}\right\}$ and $E_1, E_2, F$ be Banach spaces. Then
\[
\bigcup_{(p_1,p_2,2)\in M}\Pi_{(2;p_1,p_2)}(E_{1}, E_{2};F)\subseteq\Pi
_{bas}(E_{1}, E_{2};F).
\]
\end{theorem}
\begin{proof} Let $m\in \mathbb N$ and denote $I_{1,m}= [0,\frac{1}{2^m})$, $I_{k,m}=[\frac{k-1}{2^m}, \frac{k}{2^m})$ for $k\in \{2,3,\ldots, 2^m-1\}$ and $I_{2^m,m}=[\frac{2^m-1}{2^m}, 1]$. Hence we can define $X_{k,m}$ by means of the formula
$$\sum_{j=0}^m x_jr_j= \sum_{k=1}^{2^m}X_{k,m}\chi_{I_{k,m}}.$$
Let $x^1_j\in E_1, x^2_j\in E_2$ for $j=1,\ldots,m$. Note that
$$ \left\|\left(A(x^1_i,x^2_j )\right)_{i,j}\right\|_{Rad^{(2)}(F)}=\sup_{m\geq 0}\left(\int_0^1\left\|A\left(\sum_{i=0}^m x^1_ir_i(t), \sum_{j=0}^m x^2_jr_j(t)\right)\right\|^2 dt\right)^{1/2}.$$
On the other hand
\begin{align*}
A\left(\sum_{i=0}^m x^1_ir_i(t), \sum_{j=0}^m x^2_jr_j(t)\right)&=A\left(\sum_{k=1}^{2^m}X^1_{k,m}\chi_{I_{k,m}}(t),
\sum_{k'=1}^{2^m}X^2_{k',m}\chi_{I_{k',m}}(t)\right)\\
&=\sum_{k, k'=1}^{2^m}A\left(X^1_{k,m}\chi_{I_{k,m}}(t),X^2_{k',m}\chi_{I_{k',m}}(t)\right)\\
&=\sum_{k=1}^{2^m}A\left(X^1_{k,m},X^2_{k,m}\right)\chi_{I_{k,m}}(t).
\end{align*}
Finally observe that
\begin{align*}
 \left\|\sum_{k=1}^{2^m}A\left(X^1_{k,m},X^2_{k,m}\right)\chi_{I_{k,m}}\right\|_{L^2([0,1],F)}&=
\left(\sum_{k=1}^{2^m}\|A\left(X^1_{k,m},X^2_{k,m}\right)\|^2 2^{-m}\right)^{1/2}\\
&=
\left(\sum_{k=1}^{2^m}\left\|A\left(X^1_{k,m}2^{-m/p_1},X^2_{k,m}2^{-m/p_2}\right)\right\|^2 \right)^{1/2}\\
&\le
\pi_{(2;p_1,p_2)}(A)\cdot\sup_{\|x^*\|_{E_1^*}=1}\! \left(\sum_{k=1}^{2^m}\left|\left\langle  X^1_{k,m}2^{-m/p_1},x^*\right\rangle\right|^{p_1}\right)^{1/p_1}\\&~~~~\cdot
\sup_{\|y^*\|_{E_2^*}=1} \!\left(\sum_{k=1}^{2^m}\left|\left\langle  X^2_{k,m}2^{-m/p_2},y^*\right\rangle\right|^{p_2}\right)^{1/p_2}\\
&\le
\pi_{(2;p_1,p_2)}(A)\cdot \sup_{\|x^*\|_{E_1^*}=1}\! \left(\int_0^1 \left|\sum_{j=0}^m \langle x^1_j, x^*\rangle r_j(t)\right|^{p_1}dt\right)^{1/p_1}\\&~~~~\cdot
\sup_{\|y^*\|_{E_2^*}=1}\! \left(\int_0^1 \left|\sum_{j=0}^m \langle x^2_j, y^*\rangle r_j(t)\right|^{p_2}dt\right)^{1/p_2}\\
&\le C\pi_{2;p_1,p_2}(A) \cdot \|(x^1_j)\|_{\ell^2_w(E_1)}\cdot\|(x^2_j)\|_{\ell^2_w(E_2)}.
\end{align*}
\end{proof}

\section{Summability on $\ell_p$ spaces}

We start deriving consequences of Theorem \ref{as} in the context of $\ell_p$ spaces.

\begin{proposition}
\label{lp1}
If $E'$ has type 2 and $1\leq r\leq 2$, then
 $$\mathcal{L}(\ell_{1},E;\mathbb{K})=\Pi_{(1;2,1)}(\ell_{1}%
,E;\mathbb{K})=\Pi_{(r;2,r)}(\ell_{1}%
,E;\mathbb{K}).$$
In particular, if $1<q\leq2$ and $1\leq r\leq2$, then
 $$\mathcal{L}(\ell_{1},\ell_{q};\mathbb{K})=\Pi_{(1;2,1)}(\ell_{1}%
,\ell_q;\mathbb{K})=\Pi_{(r;2,r)}(\ell_{1}%
,\ell_{q};\mathbb{K}).$$
\end{proposition}

\begin{proof}
We only treat the case $\mathbb{K}=\mathbb{C}.$ The case $\mathbb{K}%
=\mathbb{R}$ follows from a complexification argument (see \cite{Junek-preprint, michels, Thesis} for
details).
 Let $A\in\mathcal{L}(\ell_{1},E;\mathbb{C})$. Since $E^{\prime}$ has type
$2,$ it follows from \cite[Theorem 12.10]{Diestel} that $A_{1}\in\Pi_{as}(\ell_{1};E^{\prime})$. So,
from Theorem \ref{as} it follows that%
\[
\widehat{A}\colon\ell_{2}^{w}(\ell_{1})\times\ell_{1}^{w}(E)\longrightarrow
\ell_{1}%
\]
is bounded. Hence $A\in\Pi_{(1;2,1)}(\ell_{1},E;\mathbb{C})$. The
Inclusion Theorem yields now the right hand side equality.
\end{proof}

Proposition \ref{lp1} yields for $r=1$ that
\[
\mathcal{L}(\ell_{1},\ell_{q};\mathbb{K})=\Pi_{(1;2,1)}(\ell_{1}, \ell_{q}%
;\mathbb{K})
\]
for $1\leq q\leq 2$.
Whereas, from the Inclusion Theorem and the Defant-Voigt Theorem, one also has, for any $1\le q< \infty$,
\[
\mathcal{L}(\ell_{2},\ell_{q};\mathbb{K})=\Pi_{(2;2,1)}(\ell_{2}, \ell_{q}%
;\mathbb{K}).
\]
So, by interpolation one may expect that for $1<p<2$ and $1\leq q\leq 2,$
\begin{equation}\label{eqnao}
\mathcal{L}(\ell_{p},\ell_q;\mathbb{K})=\Pi_{(p;2,1)}(\ell_{p}, \ell_q%
;\mathbb{K}).
\end{equation}

We do not know if (\ref{eqnao}) holds, but we shall prove now that it does not hold for  $q >2$ (cf. Proposition \ref{propop}) and, moreover, next we will see that we can get quite close to (\ref{eqnao}) via interpolation (cf. Theorem \ref{t2}).

%%%%%%%%%%%%%%%%%%%%%%%%%%%%%%%%%%%%%%%
%\begin{corollary}
%\label{lp2} If $1\leq  p,q\leq2$ then,
%$\mathcal{L}(\ell_{p},\ell_{q};\mathbb{K})=\Pi_{(p;2,1)}(\ell_{p},\ell_{q};\mathbb{K}).$
%\end{corollary}
%\begin{proof}
%\textcolor{red}{Once again, we only treat the case $\mathbb{K}=\mathbb{C}.$
%Let  $A\in\mathcal{L}(\ell_{p},\ell_q;\mathbb{K})$. Fix $(y_{j})\in\ell_{1}^{w}(\ell_q)$ and consider the linear mappings%
%\[
%T^{(1)}\colon\ell_{2}^{w}(\ell_{2})\rightarrow\ell_{2}\text{ and }%
%T^{(2)}\colon\ell_{2}^{w}(\ell_{1})\rightarrow\ell_{1}\text{ }%
%\]
%given by
%\[
%T^{(k)}((x_{j})_{j})=(A(x_{j},y_{j}))_{j}\text{ for }k=1,2.
%\]
%Clearly $T^{(1)}$ and $T^{(2)}$ are well-defined and continuous. Using the fact that
%$\ell_{2}^{w}(\ell_{t})=\mathcal{L}(\ell_{2};\ell_{t})$ for $t=1,2$ \cite[proof of the
%Theorem]{Pisier}, it follows that
%\[
%\ell_{2}^{w}(\ell_{p})\subset(\ell_{2}^{w}(\ell_{2}),\ell_{2}^{w}(\ell
%_{1}))_{\theta}%
%\]
%for $\frac{\theta}{2}=1-\frac{1}{p}$. So the complex interpolation method implies that%
%\begin{align*}
%T &  \colon\ell_{2}^{w}(\ell_{p})\rightarrow\ell_{p}\\
%T((x_{j})_{j}) &  =(A(x_{j},y_{j}))_{j}%
%\end{align*}
%is continuous. It follows that $A\in\Pi_{(p;2,1)}(\ell_{p},\ell_q;\mathbb{K}%
%)$.}
%\end{proof}
%

%%%%%%%%%%%%%%%%%%%%%%%%%%%%%%%%%%%%%%%%%%%%%%%

%\textcolor{red}{However,  we next prove that for $q>2$ the coincidence \eqref{eqnao} is not valid}.
The following lemma, which is less general than \cite[Theorem 2.3]{nosso}, shall be used twice later. We give a short proof for the convenience of the reader.

\begin{lemma}
\label{cotipo} Let $F$ be a Banach space, $1\le k\le n$ and assume that $E_{i}$ has finite
cotype $c_i$, $i=1,\ldots, k$. Let $p\le q$.\\
{\rm (i)} If $c_i>2$ for all $i=1,\ldots,k$, then
$$\Pi_{(p;1,\ldots,1, p_{k+1},\ldots,p_{n})}(E_{1},\ldots,E_{n};F) = \Pi
_{(q;q_{1},\ldots,q_{k}, p_{k+1},\ldots.,p_{n})}(E_{1},\ldots,E_{n};F)
$$ for any $1\le q_{i}<c_i'$, $i=1,\ldots,k$,  such that
$\sum_{i=1}^{k}\frac{1}{q_{i}}-\frac{1}{q} = k-\frac{1}{p}$.\\
{\rm (ii)} If $c_i=2$ for all $i=1, \ldots, k'$ for some $k'\le k$, then
    $$
\Pi_{(p;1,\ldots,1, p_{k+1},\ldots,p_{n})}(E_{1},\ldots,E_{n};F) = \Pi
_{(q;q_{1},\ldots,q_{k}, p_{k+1},\ldots.,p_{n})}(E_{1},\ldots,E_{n};F)
$$
 for any $1\le q_{i}\leq 2$, $i=1,\ldots,k'$ and $1\le q_{i}<c_i'$, $i=k'+1,\ldots,k$, such that
$\sum_{i=1}^{k}\frac{1}{q_{i}}-\frac{1}{q} = k-\frac{1}{p}$.
\end{lemma}

\begin{proof} In both cases the inclusion
$$\Pi_{(p;1,\ldots,1, p_{k+1},\ldots,p_n)}(E_{1},\ldots,E_{n};F) \subseteq \Pi_{(q;q_1,\ldots,q_k, p_{k+1},\ldots.,p_n)}(E_{1},\ldots,E_{n};F)$$
follows from Theorem \ref{inclusiontheorem}. Assume first that $c_i>2$ for all $i=1,\ldots , k$.
Let $A \in
\Pi_{(q;q_1,\ldots,q_k, p_{k+1},\ldots,p_n)}(E_{1},\ldots,E_{n};F)$,
$(x_j^i)_{j=1}^\infty \in \ell_{1}^w(E_i)$ for $i = 1, \ldots, k$,
and $(x_j^i)_{j=1}^\infty \in \ell_{p_i}^w(E_i)$ for $i = k+1,
\ldots, n$.  Since $E_i$ has cotype $c_i$ and $c_i<q_i'$, by \cite[Proposition
6(b)]{AB1} we know that $\ell_1^w(E_i) = \ell_{q_i'} \cdot
\ell_{q_i}^w(E_i)$, $i = 1, \ldots, k$. Hence there are
$(\alpha_j^i)_{j=1}^\infty \in \ell_{q_i'}$ and $(y_j^i)_{j=1}^\infty
\in \ell_{q_i}^w(E_k)$ such that $(x_j^i)_{j=1}^\infty = (\alpha_j^i
y_j^i)_{j=1}^\infty$, $i = 1, \ldots, k$. In this fashion,
$(\alpha_j^1 \cdots \alpha_j^k)_{j=1}^\infty \in \ell_{q_1'} \cdots
\ell_{q_k'} = \ell_{r}$, where $\frac 1r=\sum_{i=1}^k \frac{1}{q_i'}$, and $(A(y_j^1, \ldots,
y_j^k,x_j^{k+1},\ldots,x_j^n))_{j=1}^\infty \in \ell_{q}(F)$. Since
$\frac{1}{r} + \frac{1}{q} = \frac{1}{p}$ it follows that
$$ (A(x_j^1, \ldots, x_j^n))_{j=1}^\infty =(\alpha_j^1 \cdots \alpha_j^k A(y_j^1, \ldots, y_j^k, x^{k+1}_j,\ldots,x^n_j))_{j=1}^\infty
\in \ell_p(F),$$
which proves (i). If $c_i=2$ for all $i=1,\ldots,k$, then we use \cite[Proposition~6(a)]{AB1} to get  in a similar way that
$$
\Pi_{(q;2,\ldots,2, p_{k+1},\ldots.,p_{n})}(E_{1},\ldots,E_{n};F) \subseteq \Pi_{(p;1,\ldots,1, p_{k+1},\ldots,p_{n})}(E_{1},\ldots,E_{n};F).
$$
\end{proof}

Now we can prove that (\ref{eqnao}) does not hold for $q > 2$:

\begin{proposition}\label{propop}  Let $1<p<4/3 $ and $\frac{2p}{2-p}<q<\frac{p}{p-1}$. Then
$$ \Pi_{(p;2,1)}(\ell_{p}, \ell_q%
;\mathbb{K})    \neq \mathcal{L}(\ell_{p},\ell_q;\mathbb{K}).$$
\end{proposition}
\begin{proof}
%RRRRRRRRRRRRRRRRRRRRRRRRRRRRRRRR
Assume that
\begin{equation*}
\mathcal{L}(\ell_{p},\ell_q;\mathbb{K})=\Pi_{(p;2,1)}(\ell_{p}, \ell_q%
;\mathbb{K})
\end{equation*}
for some $\frac{2p}{2-p}<q<\frac{p}{p-1}$.
Since $\ell_p$ has cotype  2, one concludes from Lemma \ref{cotipo}(ii) and the assumption that, for $1/s- 1/p=1/2$ and $1/r'=1/p-1/2$,
\[
\mathcal{L}(\ell_{p},\ell_q;\mathbb{K})=\Pi_{( s; 1,1)}(\ell_{p}, \ell_q%
;\mathbb{K})=\Pi_{(1;r,1)}(\ell_{p}, \ell_q%
;\mathbb{K}).
\]
Let us see that if $T \colon \ell_{q}\longrightarrow \ell_{p'}$ is a bounded linear operator, then $\widehat T \colon \ell_1^w(\ell_q)\longrightarrow \ell_{r'}( \ell_{p'})$ is also bounded. It suffices to be  proved that for any sequences $(y_j)_{j=1}^\infty\in \ell_1^w(\ell_q)$ and $(x_j)_{j=1}^\infty \in \ell_r(\ell_p)=\ell_r'(\ell_p')^*$ the product $(\langle T(y_j),x_j\rangle)_{j=1}^\infty$ belongs to $\ell_1$. Since $\ell_r(\ell_p) \subseteq \ell_r^w(\ell_p)$, the sequence $(x_j)_{j=1}^\infty$ belongs to $\ell_r^w(\ell_p)$. Since the bilinear form
$$(x,y) \in \ell_p \times \ell_q \mapsto \langle T(y), x \rangle $$
is $(1;r,1)$-summing, it follows that $(\langle T(y_j), x_j\rangle)_{j=1}^\infty \in \ell_1$, hence, $\widehat T((y_j)_{j=1}^\infty) = (T(y_j))_{j=1}^\infty \in \ell_{r'}( \ell_{p'})$. We have just proved that $\mathcal{L}(\ell_{q};\ell_{p'})=\Pi_{(r';1)}(\ell_{q};\ell_{p'}).$ The condition on $r$ means that $r'<q<p'$ and, in this case, the formal inclusion $Id \colon \ell_{q}\longrightarrow \ell_{p'}$ would belong to $\Pi_{(r';1)}(\ell_q;\ell_{p'}).$
However, according to a result due to Carl and Bennet, independently, (see \cite[p. 209]{Diestel}), $Id\in \Pi_{(q;1)}(\ell_q;\ell_{p'})$ and $Id\notin \Pi_{(s;1)}(\ell_q;\ell_{p'})$ for $s < q$. So $q\le r'$, a contradiction that completes the proof.
\end{proof}

Let us return to the consequences of Theorem \ref{as} and Proposition \ref{lp1}. For a real number $a$, let us fix the notation $a^+ := {\rm max}\{a,0\}$. Under type/cotype assumptions we can make use of a result due to Carl and Bennet
which establishes, for $s_1\le s$, that $Id \colon \ell_{s_1}\longrightarrow \ell_s$ is $(a, 1)$-summing for $1/a =
1/s_1-(1/s-1/2)^+ $, to improve the summability of bilinear forms whenever they are restricted to a ``smaller" domain. For $1\leq s_1<s$ and $A\in {\mathcal L}(\ell_{s},E;\mathbb{K})$, let us also use $A$ to denote its restriction to $\ell_{s_1}\times E$. Henceforth the inclusion
\[
\mathcal{L}(\ell_{s},E;\mathbb{K}) \subseteq
\Pi_{(q;q_1,q_2)}(\ell_{s_1},E%
;\mathbb{K})
 \]
 means that the restriction to $\ell_{s_1}\times E$ of any continuous bilinear form defined on $\ell_s\times E$ is $(q;q_1,q_2)-$summing.

\begin{proposition}
\label{t20} Let $1\le s_1<s$ and $a$ given by $1/a=1/s_1-(1/s -1/2)^+$. Assume that  $E$ has finite cotype $r$.\\
{\rm (i)} If $s_1>2$ then $\mathcal{L}(\ell_{s},E;\mathbb{K})\subseteq \Pi_{(q;q_1,1)}(\ell_{s_1},E%
;\mathbb{K})
$ for any $1\leq q_1<s_1'$ and $q$ given by  $1/q_1-1/q=1-1/a-1/r$.\\
{\rm (ii)} If $s_1\le 2$ then $\mathcal{L}(\ell_{s},E;\mathbb{K})\subseteq \Pi_{(q;q_1,1)}(\ell_{s_1},E%
;\mathbb{K})
$ for any $1\leq q_1\le 2$ and $q$ given by  $1/q_1-1/q=1-1/a-1/r$.
\end{proposition}

\begin{proof} Let $(x_j)_{j=1}^\infty\in \ell_1^w(\ell_{s_1})$ and $(y_j)_{j=1}^\infty\in \ell_1^w(E)$. Then $(x_j)_{j=1}^\infty\in \ell_a(\ell_{s})$ for $a$ such that $1/a=1/s_1-(1/s -1/2)^+$, and,  as $E$ has cotype $r$,
$(y_j)_{j=1}^\infty\in \ell_r(E)$. This obviously leads to $(A(x_j,y_j))_{j=1}^\infty\in \ell_p,$ for $1/p=1/a+1/r$.
Consequently, $A\in \Pi_{(p;1,1)}(\ell_{s_1},E;\mathbb K)$. Since $\ell_{s_1}$ has the cotype $c=\max\{s_1,2\}$, then $A\in \Pi_{(q;q_1,1)}(\ell_{s_1},E;\mathbb K)$, where $1\le q_1<c'$ if $c=s_1>2$ or $1\leq q_1\le c$ if $c=2$, and
$1/q_1-1/q=1- 1/p$, by Lemma \ref{cotipo}.
\end{proof}

Let us now show that, using an interpolation argument, we are able to improve the result above and to get quite close to (\ref{eqnao}):

\begin{theorem}\label{t2}
%~\\ \textcolor{blue}{{\rm (a)}
%\label{lp2} If $1\leq  p,q\leq2$ then,
%$\mathcal{L}(\ell_{p},\ell_{q};\mathbb{K})=\Pi_{(p;2,1)}(\ell_{p},\ell_{q};\mathbb{K}).$\\
%{\rm (b)}}
Let $E$ be a Banach space such that $E'$ has type 2
  and $1<  p<\infty$.  Then  $ \mathcal{L}(\ell_{p},E;\mathbb{K})\subseteq \Pi_{(r;2,1)}(\ell_{p_1},E%
;\mathbb{K})
$
for  any $1\le p_1<p$  and $1/r=1/2 +   (1/p_1-1/p)p'/2$.

 \end{theorem}

\begin{proof}
%\textcolor{blue}{Let us prove (b); the proof of (a) follows the same steps.}
Again we treat only the complex case; the real case follows from a complexification argument (see \cite{Junek-preprint, michels, Thesis} for
details). Let $A \in \mathcal{L}(\ell_{p},E;\mathbb{K})$. The case $p_1=1$ follows from  Proposition \ref{lp1}. Hence $A\in\Pi_{(1;2,1)}(\ell_{1},E;\mathbb{K})$. On the other
hand, from  the Defant-Voigt theorem and the inclusion theorem we know that $\mathcal{L}(\ell_{p}%
,E;\mathbb{K})=\Pi_{(2;2,1)}(\ell_{p},E;\mathbb{K})$.
Let now $1< p_1< p$. Fix $(y_{j})_{j=1}^\infty\in\ell_{1}^{w}(E)$ and consider  the mappings%
\[
T^{(2)}\colon\ell_{2}^{w}(\ell_{p})\longrightarrow\ell_{2}\text{~ and~ }%
T^{(1)}\colon\ell_{2}^{w}(\ell_{1})\longrightarrow\ell_{1}
\]
given by
\[
T^{(k)}((x_{j})_{j=1}^\infty)=(A(x_{j},y_{j}))_{j=1}^\infty, ~k=1,2.
\]
Since $A\in \Pi_{(1;2,1)}(\ell_{1},E;\mathbb{K})\cap
 \Pi_{(2;2,1)}(\ell_{p},E;\mathbb{K})$, $T^{(1)}$ and $T^{(2)}$ are well-defined, linear and continuous. Using the fact that
$\ell_{2}^{w}(\ell_{t})=\mathcal{L}(\ell_{2};\ell_{t})$ for $t=1,p$, it follows that
\[
\ell_{2}^{w}(\ell_{p_1})\subseteq(\ell_{2}^{w}(\ell_{p}),\ell_{2}^{w}(\ell
_{1}))_{\theta}%
\]
for $\frac{1}{p_1}-\frac{1}{p}=\frac{\theta}{p'}$ (see \cite[proof of the
Theorem]{Pisier}). So the complex interpolation method implies that, for $\frac{1}{r}=\frac{1-\theta}{2}+\frac{\theta}{1}$, %
$$T  \colon\ell_{2}^{w}(\ell_{p_1})\longrightarrow\ell_{r}~,~T((x_{j})_{j=1}^\infty) =(A(x_{j},y_{j}))_{j=1}^\infty,%
$$
is continuous. It follows that $A\in\Pi_{(r;2,1)}(\ell_{p_1},E;\mathbb{K}%
)$. Now observe that
$$\frac{1}{r}-\frac{1}{2}=\frac{\theta}{2}= \frac{p'}{2}(\frac{1}{p_1}-\frac{1}{p}).$$

\end{proof}

\medskip

We finish the paper with another combination of our previous results with an argument of complex interpolation. Alternatively, this last result can also be obtained using
results from \cite{Port} and the idea of the proof of Theorem \ref{t2}.

\begin{proposition}
Let $n\geq2$ and $1<p\leq2.$ Then
\[
\mathcal{L}(\ell_{1},\overset{n-1}{\ldots},\ell_{1},\ell_{p};\mathbb{K})=
\Pi_{(r_{n};r_{n},\ldots,r_{n})}(\ell_{1},\overset{n-1}{\ldots},\ell_{1}%
,\ell_{p};\mathbb{K}),
\]
for every $1\leq r_{n}\leq\frac{2^{n-1}}{2^{n-1}-1}$.
\end{proposition}

\begin{proof}  The case $n=2$ is proved in  Proposition \ref{lp1}. From
\cite[Theorem 3 and Remark 2]{Junek} it suffices to prove the result for $r_{n}=\frac{2^{n-1}}{2^{n-1}-1}.$\\
\indent Case $n=3$ and $\mathbb{K}=\mathbb{C}$: Let
$A\in\mathcal{L}(\ell_{1},\ell _{1},\ell_{p};\mathbb{C}).$ From
 Proposition \ref{lp1} we know that
\begin{equation*}
\mathcal{L}(\ell_{1},\ell_{p};\mathbb{C})=\Pi_{(1;2,1)}(\ell_{1},\ell_{p};\mathbb{C}). \label{31Dez}%
\end{equation*}
Combining this with \cite[Corollary 3.2]{Port} we get
\[
\mathcal{L}(\ell_{1},\ell_{1},\ell_{p};\mathbb{C})=\Pi_{(1;2,1,1)}(\ell_{1},\ell_{1}%
,\ell_{p};\mathbb{C})=\Pi_{(1;1,2,1)}(\ell_{1},\ell_{1},\ell_{p};\mathbb{C}).
\]
%\begin{align*}
%\mathcal{L}(\ell_{1},\ell_{1},\ell_{p};\mathbb{C})  &  =\Pi_{(1;2,1,1)}%
%(\ell_{1},\ell_{1},\ell_{p};\mathbb{C})\\
%&  =\Pi_{(1;1,2,1)}(\ell_{1},\ell_{1},\ell_{p};\mathbb{C}).
%\end{align*}
So,%
\begin{equation}
\widehat{A}\colon\ell_{2}^{u}(\ell_{1})\times\ell_{1}^{u}(\ell_{1})\times
\ell_{1}^{u}(\ell_{p})\longrightarrow\ell_{1}~ \label{DD1}%
\end{equation}
is bounded. Combining now Proposition \ref{lp1}(ii) with
\cite[Corollary 3.2]{Port} we conclude that%
\begin{equation}
\widehat{A}\colon\ell_{1}^{u}(\ell_{1})\times\ell_{2}^{u}(\ell_{1})\times
\ell_{2}^{u}(\ell_{p})\longrightarrow\ell_{2} \label{DD2}%
\end{equation}
is bounded. So, using complex interpolation for (\ref{DD1}) and
(\ref{DD2}) we obtain that%
\[
\widehat{A}\colon\ell_{4/3}^{u}(\ell_{1})\times\ell_{4/3}^{u}(\ell_{1}%
)\times\ell_{4/3}^{u}(\ell_{p})\longrightarrow\ell_{4/3}%
\]
is bounded (this use of interpolation is based on results of
\cite{DM}, which are closely related to the classical paper
\cite{Kouba} -- further details can be found in \cite{Junek}). \\
\indent Case $n=4$ and $\mathbb{K}=\mathbb{C}$: From the case $n=3$
and \cite[Corollary 3.2]{Port} we know that
\begin{equation*}
\mathcal{L}(\ell_{1},\ell_{1},\ell_{1},\ell_{p};\mathbb{C})=\Pi_{(\frac{4}{3};1,\frac
{4}{3},\frac{4}{3},\frac{4}{3})}(\ell_{1},\ell_{1},\ell_{1},\ell_{p};\mathbb{C}).
\label{31Dez2}%
\end{equation*}
Since $\frac{4}{3} < 2$, Proposition \ref{lp1}(i) gives that
$\mathcal{L}(\ell
_{1},\ell_{p};\mathbb{C})=\Pi_{(1;\frac{4}{3},1)}(\ell_{1},\ell_{p};\mathbb{C})$.
So \cite[Corollary 3.2]{Port} implies
\[
\mathcal{L}(\ell_{1},\ell_{1},\ell_{1},\ell_{p};\mathbb{C})=\Pi_{(1;\frac{4}{3}%
,1,1,1)}(\ell_{1},\ell_{1},\ell_{1},\ell_{p};\mathbb{C}).
\]
Hence%
$$\widehat{A}    \colon\ell_{1}^{u}(\ell_{1})\times\ell_{\frac{4}{3}}^{u}%
(\ell_{1})\times\ell_{\frac{4}{3}}^{u}(\ell_{1})\times\ell_{\frac{4}{3}}%
^{u}(\ell_{p})\longrightarrow\ell_{\frac{4}{3}}~\mathrm{~and}$$
$$\widehat{A}    \colon\ell_{\frac{4}{3}}^{u}(\ell_{1})\times\ell_{1}^{u}%
(\ell_{1})\times\ell_{1}^{u}(\ell_{1})\times\ell_{1}^{u}(\ell_{p}%
)\longrightarrow\ell_{1}$$
are bounded. Using complex interpolation once more we conclude that%
\[
\widehat{A}\colon\ell_{\frac{8}{7}}^{u}(\ell_{1})\times\ell_{\frac{8}{7}}%
^{u}(\ell_{1})\times\ell_{\frac{8}{7}}^{u}(\ell_{1})\times\ell_{\frac{8}{7}%
}^{u}(\ell_{p})\longrightarrow\ell_{\frac{8}{7}}%
\]
is bounded as well. The cases $n>4$ are similar and the real case
follows, again, by complexification.
\end{proof}

\vspace{2mm}

\noindent[Oscar Blasco] Departamento de An\'alisis Matem\'atico, Universidad
de Valencia, 46.100 Burjasot - Valencia, Spain, e-mail: oscar.blasco@uv.es

\medskip

\noindent[Geraldo Botelho] Faculdade de Matem\'atica, Universidade Federal de
Uberl\^andia, 38.400-902 - Uberl\^andia, Brazil, e-mail: botelho@ufu.br

\medskip

\noindent[Daniel Pellegrino] Departamento de Matem\'atica, Universidade
Federal da Para\'iba, 58.051-900 - Jo\~ao Pessoa, Brazil, e-mail: dmpellegrino@gmail.com

\medskip

\noindent[Pilar Rueda] Departamento de An\'alisis Matem\'atico, Universidad de
Valencia, 46.100 Burjasot - Valencia, Spain, e-mail: pilar.rueda@uv.es

\end{document}